\RequirePackage{ifpdf}
\ifpdf
 	\documentclass[pdftex]{amsart}
 	\else
 	\documentclass[dvips]{amsart}
\fi

\usepackage{amsfonts,amsthm,latexsym,amsmath,amssymb,amscd,amsmath,mathrsfs, caption, subcaption, epsf}
\usepackage[enableskew]{youngtab}
\usepackage[utf8]{inputenc}
\usepackage{graphicx}

\graphicspath{{./figures/}}
\ifpdf
	\usepackage{epstopdf}		
	\DeclareGraphicsExtensions{.png,.jpg,.eps,.epsf}
\fi

\newtheorem{theorem}{Theorem}
\newtheorem{proposition}[theorem]{Proposition}
\newtheorem{lemma}[theorem]{Lemma}
\newtheorem{definition}[theorem]{Definition}
\newtheorem{corollary}[theorem]{Corollary}
\newtheorem{example}[theorem]{Example}
\newtheorem{remark}[theorem]{Remark}
\newtheorem{conjecture}[theorem]{Conjecture}

\newcommand{\defin}[1]{\textbf{\emph{#1}}}

\newcommand{\N}{\mathbb{N}}
\newcommand{\Z}{\mathbb{Z}}

\newcommand{\R}{\mathbb{R}}
\newcommand{\C}{\mathbb{C}}

\newcommand{\Tset}{\mathbf{T}}

\newcommand{\xvec}{\mathbf{x}}

\newcommand{\zvec}{\mathbf{z}}
\newcommand{\wvec}{\mathbf{w}}

\newcommand{\tvec}{\mathbf{t}}
\newcommand{\onevec}{\mathbf{1}}
\newcommand{\lambdavec}{{\boldsymbol\lambda}}
\newcommand{\alphavec}{{\boldsymbol\alpha}}
\newcommand{\muvec}{{\boldsymbol\mu}}
\newcommand{\nuvec}{{\boldsymbol\nu}}
\newcommand{\kappavec}{{\boldsymbol\kappa}}

\newcommand{\dominate}{\trianglerighteq}
\newcommand{\symS}{\mathfrak{S}}
\newcommand{\tinsert}{\boxtimes}

\DeclareMathOperator*{\length}{\mathit{l}}

\begin{document}
\title{Stretched skew Schur polynomials are recurrent}

\author[P.~Alexandersson]{Per Alexandersson}

\address{ Department of Mathematics,
   Stockholm University,
   S-10691, Stockholm, Sweden}
\email{per@math.su.se}

\begin{abstract}
We show that sequences of skew Schur polynomials obtained from stretched semi-standard Young tableaux
satisfy a linear recurrence, which we give explicitly.
Using this, we apply this to finding certain asymptotic behavior of these Schur polynomials and 
present conjectures on minimal recurrences for stretched Schur polynomials.

\end{abstract}

\maketitle

\section{Introduction}
Stretched skew tableaux, i.e.~skew semi-standard Young tableaux (SSYTs) of shape $k\muvec/k\nuvec$ for positive integers $k$,
$\muvec, \nuvec$ partitions, appear in many areas. For example, they appear naturally when studying Toeplitz matrix minors, see e.g.~\cite{Bumb02toeplitz}.
In an earlier paper \cite{Alexandersson12schur}, we found asymptotics of certain families of minors of banded Toeplitz matrices
by examining stretched skew tableaux. In this article we generalize the technique used in \cite{Alexandersson12schur} 
to explicitly give linear recurrences that skew Schur polynomials obtained from stretched semi-standard Young tableaux satisfy.

Our results appears to have close connection to systems of linear recurrences described in \cite{Hou03recurrent},
and this paper suggests that a generalization of some results in \cite{Hou03recurrent} is possible.

As an easy consequence of this paper, it follows that the number of SSYTs of shape $k\muvec/k\nuvec$
is a polynomial in $k.$ This is a well-known fact which can be proved by elementary methods.
However, it might be possible to apply the methods in this paper to prove polynomiality of stretched Kostka numbers.
This is also known, but currently requires application of non-trivial tools of different areas, see \cite{Kirillov88thebethe,King04stretched}.

As a second consequence, we prove a certain asymptotic behavior of roots of stretched skew Schur polynomials,
and conjecture the asymptotic behavior of a general system of stretched skew Schur polynomials.
Asymptotics and root location of Schur polynomials seems to be a rather unexplored topic,
except in areas where the Schur polynomials have an additional meaning, for example, as minors of Toeplitz matrices.

We use multi-index notation, i.e. $\xvec = (x_1,\dots,x_n),$ $\xvec^\alphavec = x_1^{\alpha_1}x_2^{\alpha_2} \cdots x_n^{\alpha_n}.$
The length of a vector is considered to be $n$ unless stated otherwise.
Let us now formulate the main theorem of the paper:

\begin{theorem}\label{thm:greedyskewrecurrence}
Let $n$ be a positive integer and let $\kappavec,\lambdavec, \muvec,\nuvec$ be partitions of length at most $n$
such that $\muvec \supseteq \nuvec$ and $k(\muvec-\nuvec) \supseteq \lambdavec-\kappavec$ for some positive integer $k.$
Then, for sufficiently large $r,$ the sequence $\{s_{(\kappavec + k \muvec)/(\lambdavec + k \nuvec)}(\xvec)\}_{k=r}^\infty$ 
satisfy a linear recurrence with coefficients polynomial in $x_1,x_2,\dots,x_n$. 
A characteristic polynomial for the recurrence is given by
\begin{equation}\label{eq:greedyskewrecurrence}
\chi(t) = \prod_{T \in \Tset_{\muvec/\nuvec}^n} (t - \xvec^{w(T)})
\end{equation}
where $\Tset_{\muvec/\nuvec}^n$ is the set of semi-standard Young tableaux of shape $\muvec/\nuvec$ with entries in $1,2,\dots, n$ 
and $w(T)$ is the weight of the tableau $T.$
In particular, if $\lambdavec=\kappavec=\emptyset$ we may take $r=0.$
\end{theorem}
\begin{remark}
Notice that \eqref{eq:greedyskewrecurrence} above does not necessary give the shortest possible recurrence in general. 
In Corr.~\ref{corr:minimalequation} below, we give a description of the minimal recurrence.
In Corr.~\ref{corr:asymptotic}, we use \eqref{eq:greedyskewrecurrence} for finding certain asymptotics 
of the Schur polynomials in the sequence $\{s_{(\kappavec + k \muvec)/(\lambdavec + k \nuvec)}(\xvec)\}_{k=r}^\infty.$
\end{remark}

\section{Preliminaries}
For the sake of completeness we define the basic notions in the theory of Young tableaux and Schur polynomials.
This material can be found in standard reference literature such as \cite{Macdonald79symmetric}.

\begin{definition}
A \defin{partition} $\lambdavec = (\lambda_1,\dots,\lambda_n)$ 
is a finite weakly decreasing sequence of non-negative integers; 
$$\lambda_1\geq \lambda_2 \geq \cdots \geq \lambda_n \geq 0.$$
The \defin{parts} of a partition are the positive entries 
and the number of positive parts is the \defin{length} of the partition, denoted $\length(\lambdavec)$.
The \defin{weight}, $|\lambdavec|$ is the sum of the parts.
\end{definition}
The empty partition $\emptyset$ is the partition with no parts. 
The partition $(1,1,\dots,1)$ with $k$ entries equal to $1$ is denoted $\onevec^k.$
We use the standard convention that $\lambda_i=0$ if $i>\length(\lambdavec).$
Addition and multiplication with a scalar on partitions is performed elementwise.

\begin{definition}
For partitions $\lambdavec,\muvec$ we say that $\lambdavec \supseteq \muvec$ if $\lambda_i \geq \mu_i$ for all $i.$
This is the \defin{inclusion order}. 
We also define $\lambdavec \dominate \muvec$ if $|\lambdavec|=|\muvec|$ and
$\sum_{i=1}^k \lambda_i \geq \sum_{i=1}^k \mu_i$ for all $k.$
This is the \defin{domination order}. 
\end{definition}

\subsection{Young diagrams and Young tableaux}

\begin{definition}
Let $\lambdavec \supseteq \muvec$ be partitions.
A \defin{skew Young diagram} of \defin{shape} $\lambdavec/\muvec$ is an arrangement of ``boxes'' in the plane with coordinates given by
$$\{ (i,j) \in \Z^2 | \mu_i < j \leq \lambda_i \}.$$
Here, $i$ is the row coordinate, $j$ is the column coordinate.
If $\muvec=\emptyset$ we will just refer to the shape as $\lambdavec$ and the diagram is a regular Young diagram.
\end{definition}
There are at least two other ways to draw these diagrams. In this text, the English convention is used.
Notice that the diagram of shape $\lambdavec'/\muvec'$ is the transpose of the diagram with shape $\lambdavec/\muvec.$

In this context, it will be convenient to define the \defin{skew part} of a skew diagram
as special boxes with coordinates $\{ (i,j) \in \Z^2 | 1\leq j \leq \mu_i \leq \lambda_i \}.$
We will call these boxes \defin{skew}. 
(In Fig.~\ref{fig:skewDiagramExample} there are for example seven skew boxes and six ordinary boxes.)

\begin{figure}[ht!]
  \begin{subfigure}[b]{0.48\textwidth}
   \centering
   $\young(\hfil\hfil\hfil\hfil\hfil,\hfil\hfil\hfil\hfil,\hfil\hfil,\hfil\hfil)$
  \subcaption{Diagram of shape $(5,4,2,2)$}\label{fig:diagramExample}
  \end{subfigure}
  \begin{subfigure}[b]{0.48\textwidth}
   \centering
   $\young(\blacksquare\blacksquare\blacksquare\hfil\hfil,\blacksquare\blacksquare\hfil\hfil,\blacksquare\blacksquare\hfil,\hfil)$
  \subcaption{Diagram of shape $(5,4,3,1)/(3,2,2)$}\label{fig:skewDiagramExample}
  \end{subfigure}
\caption{}
\end{figure}

\begin{definition}
A \defin{semi-standard Young tableau}\footnote{Also called column-strict tableau, or reverse plane partition} (or SSYT) 
is a Young diagram with natural numbers in the boxes,
such that each row is weakly increasing and each column is strictly increasing.
\end{definition}
We denote by $\Tset_{\lambdavec/\muvec}^n$ the set of SSYTs of shape $\lambdavec/\muvec$ with entries in $1,2,\dots,n.$
For an example of an SSYT, see Fig.~\ref{fig:ssytExample}.

\begin{figure}[ht!]
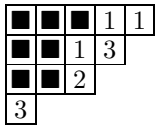

   \centering
   $\young(\blacksquare\blacksquare\blacksquare11,\blacksquare\blacksquare13,\blacksquare\blacksquare2,3)$
\caption{SSYT of shape $(5,4,3,1)/(3,2,2)$}\label{fig:ssytExample}
\end{figure}

\subsection{Schur polynomials}

\begin{definition}
Given an SSYT $T,$ with entries in $1,2\dots,n,$ we define the \defin{weight} $w(T)$ of $T$ as a vector 
$\tvec = (t_1,t_2,\dots,t_n)$ given by $t_k = \#\{ b_{ij} \in T | b_{ij} = k \}.$
Thus, $t_k$ counts the number of boxes containing the number $k.$
\end{definition}

\begin{definition}
The \defin{skew Schur polynomial} is defined as
$$s_{\lambdavec/\muvec}(\xvec) = \sum_{T \in \Tset_{\lambdavec/\muvec}^n} \xvec^{w(T)}$$
where $\xvec=(x_1,\dots,x_n).$
It can be shown that these polynomials are symmetric in $x_1,x_2,\dots, x_n$.
\end{definition}

\section{Proofs}

\subsection{Tableau insertion}
We now define an operation on pairs of SSYTs:
\begin{definition}
Given $T_1 \in \Tset_{\kappavec/\lambdavec}^n$ and  $T_2  \in \Tset_{\muvec/\nuvec}^n$ we define the \defin{tableau insertion} $T_1 \tinsert T_2$ as 
the SSYT obtained by concatenating the boxes row-wise and then sorting each row in increasing order, with respect to their content. 
The skew boxes are treated as being less than the ordinary boxes.

We also use the same notation for the corresponding operation on diagram shapes.
\end{definition}
From this definition, it is clear that the product $\tinsert$ is commutative and associative.
It is however not obvious that the result of this operation is an SSYT, 
so we prove this in the following proposition:
\begin{proposition}\label{prop:insertion}
If $T_1 \in \Tset_{\kappavec/\lambdavec}^n$ and  $T_2  \in \Tset_{\muvec/\nuvec}^n$ then
$T_1 \tinsert T_2 \in \Tset_{(\lambdavec + \muvec)/(\kappavec+\nuvec)}^n$ and 
$w(T_1 \tinsert T_2) = w(T_1) + w(T_2).$
\end{proposition}
\begin{proof}
From the definition, it is evident that the shape of $T_1 \tinsert T_2$ is $(\kappavec + \muvec)/(\lambdavec+\nuvec).$
It is also clear that the rows are weakly increasing, by construction. 
It suffices to show that the columns in $T_1 \tinsert T_2$ are strictly increasing.

Given an SSYT $T_1,$ we may view its columns $C_1,C_2,\dots, C_k$ as individual SSYTs.
Since the rows are already ordered, it is evident that $C_1 \tinsert C_2 \tinsert \cdots C_k = T_1.$
Therefore, $T_1 \tinsert T_2 = C_1 \tinsert C_2 \tinsert \cdots C_k \tinsert T_2$
and it suffices to show that $C \tinsert T_2$ is an SSYT for a general column $C.$

Let $C$ be a column with row entries $t_1, t_2, \dots, t_k$ 
where we treat skew boxes in row $i$ as having the value $n-i.$
This ensures that $t_1<t_2<\dots<t_k.$ We use the same treatment for the skew boxes in $T_2.$

It suffices to show that any two boxes in a column in adjacent rows are strictly increasing in $C \tinsert T_2.$
Let us consider rows $i$ and $i+1$ in $C \tinsert T_2.$ 
There are three cases to consider:

\textbf{Case 1:}
The numbers $t_i$ and $t_{i+1}$ are in the same column:
$$
\begin{bmatrix}
\cdots & a_1 & t_i &  a_2 & \cdots & a_m & \cdots \\
\cdots & b_1 & t_{i+1} & b_2 & \cdots & b_{m} & \cdots \\
\end{bmatrix}
$$
Since $t_i<t_{i+1},$ and all the other columns are unchanged, the columns are strictly increasing.

\textbf{Case 2:}
The number $t_i$ is to the right of $t_{i+1}$:
$$
\begin{bmatrix}
\cdots & t_i & a_1 & a_2 & \cdots & a_{m-1} & a_m & \cdots \\
\cdots & b_1 & b_2 & b_3 & \cdots & b_{m}   & t_{i+1} & \cdots \\
\end{bmatrix}
$$
The columns where strictly increasing before the insertion.
Therefore, $t_i\leq a_1 < b_1,$ $a_m < b_m \leq t_{i+1}$ and $a_j<b_j\leq b_{j+1}.$
It follows that all the columns are strictly increasing.

\textbf{Case 3:}
The number $t_i$ to the left of $t_{i+1}$:
$$
\begin{bmatrix}
\cdots & a_1 & a_2 & \cdots & a_{m-1} & a_m & t_i & \cdots \\
\cdots & t_{i+1} & b_1 & b_2 & b_3 & \cdots & b_{m} & \cdots \\
\end{bmatrix}
$$
We have that $a_j\leq t_i < t_{i+1}\leq b_k$ for $1\leq j,k\leq m,$ since the rows are increasing.
Thus, it is clear that all the columns are strictly increasing. 
It is easy to see that the result is an SSYT even if $k\neq n.$
\end{proof}

\begin{remark}
We observe that $\tinsert$ gives a monoid\footnote{Notice: this is \emph{not} the plactic monoid which is a different type of monoid structure.} structure on the set of SSYTs.
It is natural to construct the corresponding commutative \emph{ring} $\Tset_R^n$ by considering formal sums of Young tableaux
with entries in $1,2,\dots,n.$ The operation $\tinsert$ serves as multiplication, and the empty tableau $\emptyset$ acts as multiplicative identity.

For tableaux $T$ define the map $\phi(T) = \xvec^{w(T)}$ and extend it linearly to formal sums.
It is evident that $\phi(T_1 \tinsert T_2) = \phi(T_1)\phi(T_2)$ so $\phi$ acts as a ring homomorphism
from $\Tset_R^n$ to $\Z[x_1,\dots,x_n].$ 
It is therefore natural to consider $|w(\cdot )|$ as a grading on $\Tset_R^n.$
Notice that the ring $\Tset_R^n$ is finitely generated for each $n,$ a possible set of generators being all
tableaux of shape $\lambdavec/\muvec$ with $\lambda_i \leq 1$ and $\length(\lambdavec)\leq n.$ 
In other words, any tableau can be ``factored as a product of columns''.
The cancellation property also hold in $\Tset_R^n,$ 
namely if $T_1 \tinsert T = T_2 \tinsert T$ then $T_1 = T_2.$
\end{remark}

The following definition and lemmas are needed for proving the existence and to determine the constant $r$ in Thm.~\ref{thm:greedyskewrecurrence}:
\begin{definition}
Given two skew shapes $\muvec/\nuvec, \kappavec/\lambdavec,$ we say that $\muvec/\nuvec$ \defin{sits inside} $\kappavec/\lambdavec$
if every column in the diagram of shape $\muvec/\nuvec$ can be found in the diagram of shape $\kappavec/\lambdavec,$ counting multiplicities.
\end{definition}
\begin{figure}[ht!]
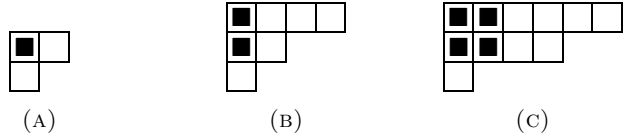
\label{fig:insideExample}
  \begin{subfigure}[b]{0.25\textwidth}
   \centering
   $\young(\blacksquare\hfil,\hfil)$
  \subcaption{}\label{fig:insideA}
  \end{subfigure}
  \begin{subfigure}[b]{0.25\textwidth}
   \centering
   $\young(\blacksquare\hfil\hfil\hfil,\blacksquare\hfil,\hfil)$
\subcaption{}\label{fig:insideB}
  \end{subfigure}
  \begin{subfigure}[b]{0.25\textwidth}
   \centering
   $\young(\blacksquare\blacksquare\hfil\hfil\hfil\hfil,\blacksquare\blacksquare\hfil\hfil,\hfil)$
\subcaption{}\label{fig:insideC}
  \end{subfigure}
\caption{Diagram (\subref{fig:insideA}) do not sit inside any of the other two diagram, 
but (\subref{fig:insideB}) sits inside (\subref{fig:insideC}).}
\end{figure}

\begin{lemma}\label{lem:sitsinside}
For every pair of skew shapes $\kappavec/\lambdavec$ and $\muvec / \nuvec$
there exists an integer $r\geq 0$ such that
$\muvec/\nuvec$ sits inside $(\kappavec + r\muvec)/(\lambdavec + r\nuvec).$
\end{lemma}
\begin{proof}
Notice that we can equivalently prove that for some $r\geq 0$ $\muvec/\nuvec$ sits inside
$\kappavec/\lambdavec \tinsert r\muvec/r\nuvec.$
The boxes in $\kappavec/\lambdavec$ ``push'' the boxes in $r\muvec/r\nuvec$ at most $\kappa_1$
places to the right when performing the tableau insertion.
If we choose $r > \kappa_1,$ then we will have $r > \kappa_1$ copies of each column in $\muvec/\nuvec,$
and therefore a tableau insertion with $\kappavec/\lambdavec$ cannot deform all of them.
This concludes the proof.
\end{proof}

\begin{lemma}\label{lem:decomposable}
Let $T \in \Tset_{\kappavec/\lambdavec}^n.$ If $\muvec/\nuvec$ sits inside $\kappavec/\lambdavec$
then there exists $T' \in \Tset_{\muvec/\nuvec}^n$ and $T'' \in \Tset_{(\kappavec-\muvec)/(\lambdavec-\nuvec)}^n$
such that $T = T' \tinsert T''.$
\end{lemma}
\begin{proof}
Since $\muvec/\nuvec$ sits inside $\kappavec/\lambdavec,$ 
we may find columns $C_1,\dots,C_k$ in $\kappavec/\lambdavec$ such that $T'=C_1 \tinsert C_2 \tinsert \cdots \tinsert C_k$
has shape $\muvec/\nuvec.$ 
The tableau $T'$ is of the correct shape, 
and deleting corresponding columns in $T$ yields a tableau $T'' \in \Tset_{(\kappavec-\muvec)/(\lambdavec-\nuvec)}^n.$
\end{proof}

We are now ready to give a proof of Thm.~\ref{thm:greedyskewrecurrence}:
\begin{proof}
We may assume that $\kappavec \supseteq \lambdavec$ since otherwise, 
choose $k$ such that $k(\muvec-\nuvec) \supseteq \lambdavec-\kappavec$ 
and take $\kappavec':= \kappavec + k \nuvec$ and $\lambdavec' := \lambdavec + k \muvec.$
Then $\kappavec' \subseteq \lambdavec'$ and the sequence 
$\{s_{(\lambdavec' + j \muvec)/(\kappavec' + j \nuvec)}\}_{j=0}^\infty$ is the same as
$\{s_{(\lambdavec + j \muvec)/(\kappavec + j \nuvec)}\}_{j=k}^\infty.$

Set $d:=|\Tset_{\muvec/\nuvec}^n|,$ which is the degree of the characteristic polynomial $\chi(t).$
By Lemma \ref{lem:sitsinside}, we may choose $r_0$ such that $\muvec/\nuvec$ sits inside $\kappavec + r_0\muvec/\lambdavec + r_0\nuvec.$ 
Let $r \geq r_0$ be arbitrary. 
It then suffices to prove that the sequence $\{s_{(\kappavec + k \muvec)/(\lambdavec + k \nuvec)}\}_{k=r}^{r+d}$
satisfy the recurrence given by $\chi(t).$

Let $\Tset_j$ be the set/formal sum of the elements in $\Tset_{(\kappavec + (j+r) \muvec)/(\lambdavec + (j+r) \nuvec)}^n.$
By Lemma \ref{lem:decomposable}, it is clear that 
\begin{equation}
\Tset_j \subset \sum_{T \in \Tset_{\muvec/\nuvec}^n} \Tset_{j-1} \tinsert T
\end{equation}
as multisets, (and equality as sets) for $j=1,2,\dots, d.$
Some tableaux appear multiple times on the right-hand side, and
these are exactly the tableaux that may be decomposed as insertions in (at least) two different ways, namely
$$\sum_{T_1,T_2 \in \Tset_{\muvec/\nuvec}^n} \Tset_{j-2} \tinsert T_1 \tinsert T_2,\quad T_1 \neq T_2.$$
Define the multisets/formal sums
$$Q_j:=\sum_{\stackrel{T_1,T_2,\dots, T_j \in \Tset_{\muvec/\nuvec}^n}{a\neq b \Rightarrow T_a \neq T_b }}  (-1)^{d-j} T_1 \tinsert T_2 \tinsert \cdots \tinsert T_j, \quad Q_0:=\emptyset.$$
Hence, $Q_j$ is, as a set, the tableaux in $\Tset_d$ that can be obtained from $\Tset_{d-j}$ by inserting $j$ different tableaux from $\Tset_{\muvec/\nuvec}^n$.

By using the principle of inclusion/exclusion, we obtain
\begin{equation}
Q_0 \tinsert \Tset_d + Q_1 \tinsert \Tset_{d-1} + Q_2 \tinsert \Tset_{d-2} +\cdots + Q_d \tinsert \Tset_{0} = 0.
\end{equation}
Application of the ring homomorphism $\phi$ to this expression now yields the desired identity.
\end{proof}

\section{Applications and further development}

\subsection{Asymptotics}

The following results are corollaries of Theorem \ref{thm:greedyskewrecurrence}:

\begin{corollary}\label{corr:minimalequation}
The sequence $\{s_{(\kappavec + k \muvec)/(\lambdavec + k \nuvec)}(\xvec)\}_{k=r}^\infty$ 
satisfy a linear recurrence, with a minimal characteristic polynomial of the form $\chi_m(t)=\prod_{\wvec \in W} (t-\xvec^{\wvec})$,
where $W \subset \N^n$ is invariant under permutations, 
i.e. $\wvec \in W \Rightarrow (w_{\sigma_1} \dots w_{\sigma_n}) \in W$ for every $\sigma \in \symS_n.$
\end{corollary}
\begin{proof}
Clearly, the roots of $\chi_m(t)$ must be a subset of the roots of \eqref{eq:greedyskewrecurrence}.
The roots of $\chi_m(t)$ be invariant under permutation of variables, since this holds for the Schur polynomials. 
This implies the invariance on $W.$ 
\end{proof}

\begin{corollary}\label{corr:asymptotic}
Let $\kappavec,\lambdavec, \muvec,\nuvec$ be partitions satisfying the conditions in Thm.~\ref{thm:greedyskewrecurrence},
with the additional condition that $\muvec \neq \nuvec.$
Set 
$$P_k(z) = s_{(\kappavec + k \muvec)/(\lambdavec + k \nuvec)}(z,\xi_2,\dots,\xi_n),\quad \xi_i\in \C, |\xi_i|=R \text{ for } i=2,3,\dots,n.$$
Define the \emph{limit set of roots} $A = \{ z\in \C | z = \lim_{k \rightarrow \infty} z_k,  P_k(z_k)=0 \}.$ 
Then $A$ is a circle with radius $R,$ possibly together with the point at the origin.
\end{corollary}
\begin{proof}
This follows from Thm.~\ref{thm:greedyskewrecurrence}, Corr.~\ref{corr:minimalequation} together with the main theorem in \cite{Beraha75limits}.
\end{proof}

\begin{example}
If $\lambdavec_n = (n,n-1,n-2,\dots,0),$ then \emph{all} roots of $s_{k \lambdavec_n}(t,\onevec^n)=0$
lie on the unit circle, for every $n,k.$
\end{example}

The following conjecture is a generalization of Corr.~\ref{corr:asymptotic}:
\begin{conjecture}
Let $1 \leq j \leq n$ and $\kappavec_i,\lambdavec_i, \muvec_i,\nuvec_i$, $1\leq i \leq j,$ 
be partitions satisfying the assumptions in Thm.~\ref{thm:greedyskewrecurrence}.

Let $x_i\in \C,$ $|\xi_i|=R$ for $i=j+1,\dots,n$ and define
\begin{equation}
P_k^i(z_1,\dots,z_j) = s_{(\kappavec_i + k \muvec_i)/(\lambdavec_i + k \nuvec_i)}(z_1,z_2,\dots,z_j,\xi_{j+1},\dots,\xi_n), 1\leq i \leq j.
\end{equation}
Set
$$A = \{ \zvec \in \C^j | \zvec = \lim_{k \rightarrow \infty} \zvec_k,  P_k^1(\zvec_k)=P_k^2(\zvec_k)=\dots =P_k^j(\zvec_k) =0 \}.$$
Then, under some mild non-degeneracy conditions on the partitions,
\begin{equation}
A = Z \cup
\begin{cases}
\{ R(e^{i\theta_1},e^{i\theta_2},\dots,e^{i\theta_j}) | \theta_1,\theta_2,\dots,\theta_j \in \R \} \text{ if } j<n, \\
\{ R(e^{i\theta_1},e^{i\theta_2},\dots,e^{i\theta_j}) | R,\theta_1,\theta_2,\dots,\theta_j \in \R \} \text{ if } j=n.
\end{cases}
\end{equation}
where $Z$ is either $\emptyset$ or the set consisting of the origin.
\end{conjecture}
\begin{remark}
This is true in a slightly modified special case. Multivariate Chebyshev polynomials may be defined 
as certain polynomials $P_k^i$ as above, with an appropriate change of variables. The support of the orthogonality measure 
is the image of $A$, under the same mapping as the change of variables. 
See \cite{Beerends91chebyshev} for the connection between multivariate Chebyshev polynomials of the second kind, and Schur polynomials.
\end{remark}

\subsection{Kostka coefficients}

The recurrence \eqref{eq:greedyskewrecurrence} is in some cases not the shortest possible.
For some applications, this is not a problem but it is not completely satisfying.
Below we conjecture the shortest possible recurrence. We need the following definitions:

\begin{definition}
Given a partition $\lambdavec,$ we define the \defin{monomial symmetric polynomial} $m_\lambdavec$ as
\begin{equation}
m_\lambdavec = \sum_{\wvec} \xvec^\wvec,
\end{equation}
where the sum is taken over \emph{distinct} permutations of $\lambdavec.$ 
\end{definition}
Note that $m_{k \lambdavec}(x_1,x_2,\dots,x_n) = m_{\lambdavec}(x_1^k,x_2^k,\dots,x_n^k)$ by definition.

\begin{definition}
The Kostka coefficient $K_{\lambdavec/\muvec,\wvec}$ is the number of tableaux of shape $\lambdavec/\muvec$ with weight $\wvec.$
\end{definition}
It is well-known that $K_{\lambdavec/\muvec,\wvec} = K_{\lambdavec/\muvec,\overline{\wvec}}$ where $\overline{\wvec}$ is the vector obtained from $\wvec$
by rearranging the elements as a partition, in decreasing order.
It is evident that $K_{\lambdavec/\muvec,\wvec} = 0$ if $|\wvec| \neq |\lambdavec|-|\muvec|$.
The Kostka numbers and the monomial symmetric polynomials are related by:
\begin{equation}
s_{\lambdavec/\muvec}(\xvec) = \sum_{\wvec} K_{\lambdavec/\muvec,\wvec} m_\wvec(\xvec), 
\end{equation}
where the sum is taken over all partitions $\wvec.$

We now give a hands-on application of tableau insertion and Kostka coefficients:
\begin{proposition}
If $K_{\lambdavec/\muvec,\wvec}>0$ then $K_{k\lambdavec/k\muvec,k\wvec}>0$ for any integer $k>0.$
\end{proposition}
\begin{proof}
Let $T$ be a tableau of shape $\lambdavec/\muvec$ with weight $\wvec.$
Then the $k-$th power $T \tinsert \cdots \tinsert T$ is a tableau with shape $k\lambdavec/k\muvec$ and weight $k\wvec$,
and hence $K_{k\lambdavec/k\muvec,k\wvec}>0.$
\end{proof}
\begin{remark}
In fact, $K_{\lambdavec,\wvec}>0 \Leftrightarrow K_{k\lambdavec,k\wvec}>0$ and this is known as
Fulton's K-saturation conjecture. 
Its proof is given by \cite{King04stretched}, which uses the K-hive model machinery.
\end{remark}

We now give a conjectural sharper version of Thm.\ref{thm:greedyskewrecurrence}:
\begin{conjecture}\label{conj:skewrecurrence}
Let $\kappavec,\lambdavec, \muvec,\nuvec$ be partitions of length at most $n,$ 
such that $\muvec \supseteq \nuvec$ and $k(\muvec-\nuvec) \supseteq \kappavec-\lambdavec$ for some positive integer $k.$
Set
\begin{eqnarray*}
W &=& \{\wvec \in \N^n | K_{\muvec/\nuvec, \wvec} > 0 \wedge \overline{\wvec} \dominate \overline{\muvec-\nuvec} \}.
\end{eqnarray*}
Then, for sufficiently large $r,$ the sequence $\{s_{(\kappavec + k \muvec)/(\lambdavec + k \nuvec)}(\xvec)\}_{k=r}^\infty$ 
satisfy a linear recurrence with the minimal characteristic polynomial 
\begin{equation}\label{eq:skewcharacteristic}
\chi(t) = \prod_{\wvec \in W} (t - x^\wvec).
\end{equation}
\end{conjecture}
In Thm.~\ref{thm:greedyskewrecurrence}, it is obvious how to interpret the 
coefficients in the linear recurrence as certain tableaux insertions, mapped under the ring homomorphism $\phi.$
However, in the conjectures above, it is not even clear if such interpretation exists.

\section{Acknowledgement}
The author would like to thank Prof.~B.~Shapiro for helpful comments on the text.


\end{document}